\documentclass[11 pt]{article}

\usepackage{amsmath,amsthm,amsfonts,amssymb,csquotes,epsfig, titlesec, epsfig, float, caption,wrapfig, mathrsfs}
\usepackage[hidelinks]{hyperref}
\usepackage{tikz-cd}
\usepackage{tikz}
\usepackage{tikz-3dplot}
\usepackage{xcolor}
\usepackage{verbatim}
\usepackage{accents}
\usepackage{marvosym}
\usepackage{multirow}
\usepackage{multicol}
\usepackage{listings}
\usepackage[citestyle=alphabetic,bibstyle=alphabetic,backend=bibtex, maxbibnames=10,minbibnames=5]{biblatex}
\usepackage[]{enumerate}
\usepackage[american]{babel}

\usetikzlibrary{calc, decorations.pathreplacing,shapes.misc}
\usetikzlibrary{decorations.pathmorphing}
\usepackage[left=1in,top=1in,right=1in]{geometry}
\usepackage[capitalize, noabbrev]{cleveref}
\usetikzlibrary{shapes.geometric}

\newtheorem{thm}{Theorem}[section]
\newtheorem{lem}[thm]{Lemma}

\theoremstyle{remark} 
\newtheorem{rem}[thm]{Remark}

 \crefname{rem}{Remark}{Remarks}
 \Crefname{rem}{Remark}{Remarks}
\newtheorem{ex}[thm]{Example}

\theoremstyle{definition} 
 
\newtheorem{definition}[thm]{Definition} 

\titleformat*{\section}{\normalsize \bfseries \filcenter}
\titleformat*{\subsection}{\normalsize \bfseries }
\newtheorem{mainthm}{Theorem}
\Crefname{mainthm}{Theorem}{Theorems}
\newtheorem{maincor}[mainthm]{Corollary}
\Crefname{maincor}{Corollary}{Corollaries}

\Crefname{mainquestion}{Question}{Questions}

\Crefname{mainconj}{Conjecture}{Conjectures}

\crefformat{equation}{(#2#1#3)}

\makeatletter
\def\namedlabel#1#2{\begingroup
   \def\@currentlabel{#2}
   \label{#1}\endgroup
}
\makeatother

\DeclareFontFamily{U}{mathx}{}
\DeclareFontShape{U}{mathx}{m}{n}{<-> mathx10}{}
\DeclareSymbolFont{mathx}{U}{mathx}{m}{n}
\DeclareMathAccent{\widehat}{0}{mathx}{"70}
\DeclareMathAccent{\widecheck}{0}{mathx}{"71}

\renewbibmacro{in:}{}

\usetikzlibrary{decorations.markings}
\usepackage{amsmath}
\usepackage{mathtools}
\usepackage{gensymb}

\title{\normalsize \textbf{Rational points of fixed denominator in real toric arrangements}}
\author{\normalsize Andrew Hanlon and Davis Painter}
\date{}

\newcommand{\wt}{\widetilde}

\newcommand{\RR}{\mathbb R}
\newcommand{\ZZ}{\mathbb Z}

\newcommand{\NN}{\mathbb N}
\newcommand{\PP}{\mathbb P}

\DeclareMathOperator{\lcm}{lcm}

\newcommand{\Addresses}{{% additional braces for segregating \footnotesize
  \bigskip
  \footnotesize

  \noindent A.~Hanlon, \textsc{Department of Mathematics, University of Oregon}\par\nopagebreak
  \noindent \textit{E-mail address}: \hyperlink{mailto:ahanlon@uoregon.edu}{\texttt{ahanlon@uoregon.edu}}

  \medskip

  \noindent D.~Painter, \textsc{Dartmouth College}\par\nopagebreak
  \noindent \textit{E-mail address}: \hyperlink{mailto:davis.g.painter.26@dartmouth.edu}{\texttt{davis.g.painter.26@dartmouth.edu}}

  \medskip

}}

\begin{document}

\maketitle

\begin{abstract}
    We give a sufficient condition on a positive integer $m$ for every stratum of a given real toric hyperplane arrangement to contain a rational point of denominator $m$. As a consequence, we give a sufficient condition on $m$ for the degree $m$ Frobenius pushforward of the structure sheaf on a smooth toric variety to contain all possible summands in the Picard group.
\end{abstract}

\section{Introduction}

Given a set of vectors $A = \{ v_1, \hdots, v_k \} \subset \ZZ^n$, we have functions 
\begin{equation} \label{eq:functions}
    f_i\colon \RR^n/\ZZ^n \to \RR/\ZZ 
\end{equation}
given by $f_i(x) = v_i \cdot x$ where $\cdot$ is the dot product. 
Further, we obtain a real toric hyperplane arrangement given by the linear subtori $h_i = f_i^{-1}(0).$
In this paper, we study the stratification $\mathcal{S}_A$ of the torus induced by this hyperplane arrangement, which we define precisely in \cref{sec:strata}. 
Results on enumerating the strata and other combinatorial properties of real toric hyperplane arrangements have been obtained in \cite{ehrenborg2009affine, lawrence2011enumeration, chandrasekhar2017face, bergerova2023symmetry}. 
There is also interest in and a close relation to complex toric arrangements as recently studied in \cite{de2005geometry, moci2012tutte, d2012ehrhart,d2015minimality}. 

Here, we will focus on finding rational points of bounded denominator in all the strata. 
More precisely, for $m \in \NN$, let $L_m$ be the image of $\frac{1}{m} \ZZ^n$ in $\RR^n/\ZZ^n$. Moreover, assuming $A$ contains a basis of $\RR^n$, let
\begin{equation} \label{eq:determinantlcm} 
D_A = \lcm\left\{ \left| \det\begin{pmatrix} \vert &  & \vert \\ v_{i_1} & \hdots & v_{i_n} \\ \vert &  & \vert \end{pmatrix} \right| \neq 0 : i_j \in \{1, \hdots, k\} \right\} 
\end{equation}
be the least common multiple of all absolute values of nonzero determinants of matrices formed by a set of $n$ vectors in $A$. 
Our main result is a sufficient condition for $L_m$ to intersect every strata.

\begin{mainthm} \label{thm:main}
   Suppose that $A$ contains a basis of $\RR^n$. If $m$ is a multiple of $D_A$ and $m \geq (n+1)D_A$, then $L_m \cap S \neq \emptyset$ for all $S \in \mathcal{S}_A$. 
\end{mainthm}

We note that \cref{thm:main} is sharp in the sense that there exist examples such that $L_m \cap S \neq \emptyset$ for all $S \in \mathcal{S}_A$ if and only if $m = \ell D_A$ with $\ell \geq n+1$ (see \cref{ex:p1}).
When $A$ does not contain a basis of $\RR^n$, we can reduce to a lower-dimensional case of \cref{thm:main} by working with the span of $A$ and the appropriate lattice.

\cref{thm:main} is of particular interest to us due to recent results on toric varieties. 
In an influential note \cite{bondal2006derived}, Bondal indicated that the stratification $\mathcal{S}_A$ can be used to study homological properties of a toric variety $X_\Sigma$ such that $A$ is the set of primitive generators of the fan $\Sigma$. 
For example, this perspective has been fruitful for studying homological mirror symmetry for toric varieties \cite{fang2011categorification,kuwagaki2020nonequivariant,hanlon2022aspects} and for understanding generation of the derived category \cite{favero2023rouquier, hanlon2023resolutions}.
As we explain in more detail in \cref{sec:toricvar}, there is a map from $L_m$ to the summands of the pushforward of $\mathcal{O}_{X_\Sigma}$ by the degree $m$ toric Frobenius map $F_m$. 
This map is surjective when $L_m \cap S \neq \emptyset$ for all $S \in \mathcal{S}_A$. Thus, we obtain the following from \cref{thm:main}.

\begin{maincor} \label{cor:main}
    Suppose that $X_\Sigma$ is a smooth toric variety with no torus factors and $A$ is the set of primitive generators of $\Sigma(1)$. 
    If $m = \ell D_A$ with $\ell \geq \dim(X_\Sigma) +1$, then every possible summand of the pushforward of $\mathcal{O}_{X_\Sigma}$ under toric Frobenius appears in $(F_m)_* \mathcal{O}_{X_\Sigma}$. 
\end{maincor}

The condition that $X_\Sigma$ has no torus factors is a simple way to impose that $A$ contains a basis of $\RR^n$.
As with \cref{thm:main}, it is relatively straightforward to generalize \cref{cor:main} to other smooth toric varieties.
We also expect that the condition that $X_\Sigma$ is smooth can be relaxed, but we include it as results in the literature on the splitting of toric Frobenius are written under this assumpotion (see \cref{rem:smooth}). 

The remainder of this paper is organized as follows. \cref{sec:strata} more carefully defines $\mathcal{S}_A$ and establishes further notation. The proof of \cref{thm:main} is confined to \cref{sec:proof} where it is broken down into two lemmas that are proved with elementary arguments. Our application to toric varieties (\cref{cor:main}) is explained in \cref{sec:toricvar}. Finally, \cref{sec:examples} contains some examples and further questions. 

\subsection{Acknowledgements}
We are grateful to Daniel Erman and Jeff Hicks for useful conversations. This paper builds on undergraduate research work completed for credit at Dartmouth College by the second author, and we thank Dartmouth and the Department of Mathematics for supporting this activity.

Both authors were supported by NSF grant DMS-2404882, which was transferred to NSF grant DMS-2549013.

\section{Toric hyperplane stratifications} \label{sec:strata}

In this section, we precisely define the stratification $\mathcal{S}_A$ of the torus $\RR^n/\ZZ^n$ that we associate to $A = \{ v_1, \hdots, v_k \} \subset \ZZ^n$. 
Recall from \eqref{eq:functions} that we define $f_i$ to be the circle-valued function on the torus given by $f_i(x) = v_i \cdot x$ for $i = 1,\hdots, k$ and consider the hyperplanes $h_i = f_i^{-1}(0)$. 
In other words, the vectors $v_i$ are normal vectors to the toric hyperplane arrangement given by the $h_i$.

Further, for any $I \subset \{1, \hdots, k\}$, we can consider the closed subset
\[ h_I = \{ x \in \RR^n/\ZZ^n : f_i(x) = 0 \text{ for all } i \in I \} \]
and set $S_{I, 1}, \hdots, S_{I, p_I}$ to be the connected components of 
\[ h_I \setminus \left( \bigcup_{j \not \in I} h_j \right) \]
giving us a stratification
\[ \RR^n/\ZZ^n = \bigcup_{I \subset \{1, \hdots, k\}} \bigsqcup_{j = 1}^{p_I} S_{I, j} \]
by the locally closed strata $S_{I, j}$.
Note that $h_\emptyset = \RR^n/ \ZZ^n$ so the top-dimensional strata are of the form $S_{\emptyset, j}$.
The collection of nonempty locally closed subsets obtained in this way precisely describes the stratification induced by the toric hyperplane arrangement.

\begin{definition}
    We call
    \[ \mathcal{S}_A = \{ S_{I,j} \} \setminus \{ \emptyset \} \]
    the stratification of $\RR^n/\ZZ^n$ induced by the set of vectors $A$. 
    When the precise nature of the stratum is unimportant, we will often drop the subscripts and write $S \in \mathcal{S}_A$ to denote one of the locally closed strata.
\end{definition}

    Note that each subtorus $h_i$ can equivalently be defined as the image of the family of hyperplanes $\wt h_i = \{ x \in \RR^n: v_i \cdot x  \in \ZZ \}$ under the quotient map from $\RR^n$ to $\RR^n/\ZZ^n$. 
    We can also describe the strata as the image of locally closed subsets of $\RR^n$ as follows. 
    Given $I \subset \{1, \hdots, k \}$ and $u \in \ZZ^k$, define the set 
    \begin{equation} \label{eq:liftedpolytope}
    \wt S_{I,u} = \{ x \in \RR^n : x \cdot v_i = u_i \text{ for } i \in I \text{ and } u_j < x \cdot v_j < u_j + 1 \text{ for } j \not \in I \}
    \end{equation}
    Then, we have that $\RR^n$ is the disjoint union of the locally closed polytopes $\wt S_{I,u}$ and we obtain each $S_{I,j}$ as the image under the quotient map of some $\wt S_{I,u}$. 
    Typically, each $S_{I,j}$ is in fact the image of infinitely many $\wt S_{I,u}$. 
    We will refer to each $\wt S_{I,u}$ which maps to $S_{I,j}$ as a lift of $S_{I,j}$. 

Examples of these stratifications and figures depicting those examples appear in \cref{sec:examples}. 

\section{Proof of the main theorem} \label{sec:proof}

We will deduce \cref{thm:main} from two lemmas. 
The first is an exercise in linear algebra that allows us to rescale all lifts of strata to lattice polytopes. Here, we use the term lattice polytope to mean an intersection of half-spaces whose vertices lie in $\ZZ^n$.

\begin{lem} \label{lem:rescale}
    Assume that $A$ contains a basis of $\RR^n$ and let $D_A$ be as in \eqref{eq:determinantlcm}. 
    If $S \in \mathcal{S}_A$ and $\Delta = \overline{\wt{S}}$ is the closure of any lift $\wt{S}$ of $S$ as in \eqref{eq:liftedpolytope}, then the rescaling $D_A \Delta$ is a bounded lattice polytope.
\end{lem}
\begin{proof} 
We first note that $\Delta$ is a polytope by definition. 
Since $A$ contains a basis of $\RR^n$, every lift of a stratum is contained within a parallelepiped of volume given by the determinant of such a basis. 
Thus, $\Delta$ is bounded. 

It remains to check that if $w \in \Delta$ is a vertex, then $D_A w \in \ZZ^n$. 
Note that $w$ must lie on at least $n$ hyperplanes $\tilde{h}_i$. 
Assume, without loss of generality, that $v_1, \hdots, v_n$ are linearly independent and there exists $k_1, \hdots, k_n \in \ZZ$ such that 
\[ v_i \cdot w = k_i \]
for $i = 1, \hdots, n$. 
In other words, we have $Bw = k$ where $k \in \ZZ^n$ has $i$th component equal to $k_i$ and $B$ is the matrix whose $i$th row is $v_i$.
We then have
\[ w = B^{-1} k = \frac{1}{\det(B)} C^T k\]
where $C$ is the matrix of cofactors of $B$. 
In particular, $C$ has integer entries.
Since $D_A/\det(B)$ is an integer, we see that $D_A w \in \ZZ^n$ as needed.
\end{proof} 

Our second lemma gives the minimum amount of rescaling required to find an interior lattice point, i.e., a point in $\ZZ^n$, in any lattice polytope. 
The result follows from Ehrhart reciprocity and possibly other results about enumerating lattice points in lattice polytopes, but we include an elementary proof. 

\begin{lem} \label{lem:interiorpt}
    Suppose that $\Delta$ is a $p$-dimensional bounded lattice polytope in $\RR^n$. Then $k\Delta$ contains a lattice point in its relative interior for all $k \geq p + 1$. 
\end{lem}
\begin{proof} 
We will reduce to the case where $\Delta$ is a simplex as follows. Let $w_0, \hdots, w_p$ be any $p+1$ vertices of $\Delta$ whose convex hull $\Delta'$ is $p$-dimensional. 
Note that such a collection of vertices must exist since $\Delta$ is a $p$-dimensional bounded polytope.
Every facet of $\Delta'$ contains all but exactly one of the $w_j$ so we can choose outward normal vectors to these facets $\alpha_0, \hdots, \alpha_p$ such that $\alpha_j \cdot w_i = c_j$ if $i \neq j$ and $\alpha_j \cdot w_i < c_j$ if $i = j$ for some constants $c_j$.  

We then observe that 
\[ \alpha_j \cdot (w_0 + \hdots + w_p) < (p+1) c_j \]
for $j = 0, \hdots, p$.
Thus, $w_0 + \hdots + w_p$ is a lattice point in the relative interior of $(p+1)\Delta'$ and hence $(p+1) \Delta$.
It follows that $k \Delta$ contains a lattice point in its relative interior for all $k \geq p + 1$.
\end{proof}

We are now ready to prove \cref{thm:main}.

\begin{proof}[Proof of \cref{thm:main}] Let $S \in \mathcal{S}_A$ be any stratum. 
Further, let $\tilde{S}$ be any lift of $S$ and let $\Delta$ be the closure of $\tilde{S}$.
Note that $L_m \cap S \neq \emptyset$ if and only if $\Delta$ contains a rational point with denominator $m$ in its relative interior.
That is, we need $m \Delta$ to contain a lattice point in its relative interior. 
By \cref{lem:rescale}, $D_A \Delta$ is a bounded lattice polytope of dimension less than or equal to $n$.
Thus, by \cref{lem:interiorpt}, $k D_A \Delta$ contains an interior lattice point in its relative interior for $k \geq n +1$. 
\end{proof}

\section{Application to toric varieties} \label{sec:toricvar}

In this section, we explain how \cref{cor:main} follows from \cref{thm:main}.
Outside of this section, we have made things concrete by considering $A \subset \ZZ^n$ and working with the dot product, but to be more consistent with the literature on toric varieties, we will use a more abstract set up in this section. 

A toric variety $X_\Sigma$ (over any fixed ground field) can be constructed from a fan $\Sigma$ in $N_\RR = N \otimes_\ZZ \RR$ where $N$ is a lattice.
Any such $X_\Sigma$ admits a family of toric morphisms
\[
F_m \colon X_\Sigma \to X_\Sigma 
\]
indexed by $m \in \NN$ called the toric Frobenius morphisms. 
On the dense torus $N \otimes_\ZZ \Bbbk^*$, $F_m$ is simply the $m$th power map.
When $X_\Sigma$ is smooth, the pushforward of the structure sheaf $(F_m)_* \mathcal{O}_{X_\Sigma}$ always splits as a direct sum of line bundles as originally shown in \cite{thomsen2000frobenius} (see also \cite{bogvad1998splitting, achinger2015characterization}). 
Moreover, for sufficiently large and divisble $m$, all possible summands appear in $(F_m)_* \mathcal{O}_{X_\Sigma}$. 
Our \cref{cor:main} gives a precise sufficient condition on exactly how large and divisible $m$ must be in order for all possible summands to appear as we now discuss. 

Let $M$ be the dual lattice to $N$. 
Bondal observed \cite{bondal2006derived} (as spelled out in \cite[Corollary 5.5]{hanlon2023resolutions}) that 
\begin{equation} \label{eq:pushforwardeq}
    (F_m)_* \mathcal{O}_{X_\Sigma} = \bigoplus_{x \in M/m M} \mathcal{O}_{X_\Sigma} \left( \sum_{\rho \in \Sigma(1)} \left \lfloor \frac{ -\langle x , u_\rho \rangle }{m} \right \rfloor D_\rho \right) 
\end{equation} 
where $u_\rho$ is the primitive generator of the ray $\rho \in \Sigma(1)$, $D_\rho$ is the toric divisor corresponding to $\rho$, and $\langle x, u_\rho \rangle$ denotes the pairing of any lift of $x$ to $M$ with $u_\rho$. 
As a consequence of \eqref{eq:pushforwardeq}, all possible summands have the form
\begin{equation} \label{eq:summand}
    \mathcal{O}_{X_\Sigma} \left( \sum_{\rho \in \Sigma(1)} \lfloor -\langle x , u_\rho \rangle \rfloor D_\rho \right)
\end{equation}
for some $x \in M_\RR/M$ where $M_\RR = M \otimes_\ZZ \RR$. 
Further, we observe that \eqref{eq:summand} is constant on each $S \in \mathcal{S}_A$ where $A$ is the set of primitive generators of $\Sigma$ and $\mathcal{S}_A$ is now a stratification of $M_\RR/M$. 
Thus, the assigment $ x \mapsto$ \eqref{eq:summand} gives a surjective map from $\mathcal{S}_A$ to the summands of toric Frobenius in the Picard group of $X_\Sigma$ where each $x \in L_m$ maps to a summand of $(F_m)_* \mathcal{O}_{X_\Sigma}$ by \eqref{eq:pushforwardeq}. 
Therefore, we now see that \cref{cor:main} follows from \cref{thm:main}

We conclude this section with a few remarks. 

\begin{rem} \label{rem:smooth}
    The assumption that $X_\Sigma$ is smooth is likely unnecessary, and we only impose this assumption as the splitting result \eqref{eq:pushforwardeq} is only proved explicitly in the literature when $X_\Sigma$ is smooth.
    In general, we expect that $(F_m)_* \mathcal{O}_{X_\Sigma}$ still splits into a sum of reflexive sheaves of rank one given by \eqref{eq:pushforwardeq}. 
    Without that generalization, \cref{thm:main} still implies that all line bundles of the form \eqref{eq:summand} (the Thomsen collection in the sense of \cite[Section 2.3]{hanlon2023resolutions}) can be obtained using rational $x$ with denominator $m$ satisfying the conditions of \cref{thm:main}. 
\end{rem}

\begin{rem} \label{rem:git}
    Given a finite set $A \subset N$, there are often many smooth toric varieties for which the primitive generators of the fan are a subset of $A$. 
    If $m$ satisfies the conditions of \cref{thm:main}, the degree $m$ Frobenius pushforward of the structure sheaf on any of these toric varieties will contain all possible summands. 
    For example, $A$ determines a torus action on $\mathbb{A}^{\Sigma(1)}$ whose GIT quotients are all toric varieties with fans generated by a subset of $A$.
\end{rem}

\begin{rem} As noted in \cite{achinger2015characterization}, the summands of $(F_m)_* \mathcal{O}_{X_\Sigma}$ are lattice points in $\frac{m-1}{m} Z$ where $Z$ is a certain zonotope in the real Picard group of $X_\Sigma$. 
Moreover, the possible summands are in bijection with the lattice points of $Z$ that lie in the star of $0 \in Z$ by \cite[Proposition 5.14]{hanlon2023resolutions}.
Thus, we obtain the curious conclusion that if $m$ satisfies the conditions of \cref{thm:main}, then $\frac{m-1}{m} Z$ contains all lattice points in the star of $0$.
\end{rem}

\section{Examples and further questions} \label{sec:examples}

In this concluding section, we give a few examples to illustrate \cref{thm:main} and pose several questions that would be interesting to study further.

\begin{ex} \label{ex:p1}
    Set $A = \{ e_1, \hdots, e_n,  -e_1 - \hdots - e_n \}$ where $e_j$ are the standard basis vectors of $\RR^n$. 
    When $n=2$, there are two strata of dimension two, three strata of dimension one, and one stratum of dimension zero as depicted in \cref{fig:p1}.
    In any dimension, one of the top-dimensional strata is the interior of the simplex $\Delta$ with vertices $0, e_1, \hdots, e_n$. 
    As $D_A = 1$ and $\Delta$ has no interior rational points of denominator less than $n+1$ (that is, $L_{k} \cap \mathrm{int}(\Delta) = \emptyset$ when $k \leq n$), we see that the converse of \cref{thm:main} holds in this case.
    The set $A$ in this example consists of the primitive generators of the rays of the toric variety $X_\Sigma = \PP^n$. 
\end{ex}

\begin{ex} \label{ex:morecomplicated}
    Consider $A = \{ e_1, e_2, e_1 + e_2, -e_1 + 2e_2 \}$ where $e_1,e_2$ are the standard basis vectors of $\RR^2$.
    In this example, there are six strata of dimension two, ten strata of dimension one, and four strata of dimension zero. 
    See \cref{fig:morecomplicated}.
    Here, $D_A = 6$, but we can see that $L_{12} \cap S \neq \emptyset$ for all $S \in \mathcal{S}_A$, which is a smaller denominator than \cref{thm:main} produces.
\end{ex}

In light of \cref{ex:p1} and \cref{ex:morecomplicated}, it is natural to ask when the condition in \cref{thm:main} is also necessary or when a smaller $m$ guarantees that $L_m \cap S \neq \emptyset$ for all $S \in \mathcal{S}_A$. 
When $n=1$, it is straightforward to check that the condition in \cref{thm:main} is always necessary. 
In general, as indicated by the proof of \cref{thm:main}, it seems natural to approach this problem by determining which lattice polytopes have rescalings with no lattice points in their relative interiors and when such polytopes can appear as closures of lifts of strata.
It would be interesting to study this question further and to understand what information beyond $D_A$ is needed to determine a necessary and sufficient condition on $m$. 

In fact, to get a better bound for the purposes of \cref{cor:main}, we can consider larger strata as we now explain. 
The assignment \eqref{eq:summand} of a line bundle on $X_\Sigma$ to a point in $M_\RR/M$ is constant on larger sets than the strata in $\mathcal{S}_A$, and we could simply define a new stratification by the level sets of \eqref{eq:summand}.
For instance, in \cref{ex:p1}, the three one-dimensional strata all become part of the same two-dimensional stratum under this definition.
Using these larger strata would still give a simultaneous bound for all toric varieties that arise as quotients of the same action on affine space as in \cref{rem:git}.
Thus, we might hope for a smaller sufficient condition in \cref{cor:main}, and it does not appear that the methods employed in this paper are particularly well-suited for determining such a bound.

Finally, we end with a more amusing question. 
In an earlier and more involved approach to \cref{thm:main} with $n =2$, we attempted to find a bound on $m$ by controlling the area of all the strata that appear.
In many examples, we found that as long as $A$ contains three vectors none of which is a scalar multiple of another (still assuming $n=2$) then the stratum of smallest area is a triangle. 
Does this always hold?
If so, is there a higher dimensional generalization?

\begin{figure}[!b] 
\centering 

\begin{tikzpicture}[scale=1.2]

    \draw 
        (-4,0.75) -- (-4,4.25);
    \draw 
        (-3,0.75) -- (-3,4.25);
    \draw 
        (-2,0.75) -- (-2,4.25);
    \draw 
        (-1,0.75) -- (-1,4.25);
    \draw 
        (-4.25,1) -- (-0.75,1);
        \draw 
        (-4.25,2) -- (-0.75,2);
    \draw 
        (-4.25,3) -- (-0.75,3);
    \draw 
        (-4.25,4) -- (-0.75,4);
    \draw 
        (-4.25,4.25) -- (-0.75,0.75);
    \draw 
        (-3.25,4.25) -- (-0.75,1.75);
    \draw 
        (-2.25,4.25) -- (-0.75,2.75);
    \draw 
        (-4.25,3.25) -- (-1.75,0.75);
    \draw 
        (-4.25,2.25) -- (-2.75,0.75);
    \draw
        (-4.25,1.25) -- (-3.75,0.75);
    \draw
        (-1.25,4.25) -- (-0.75,3.75);
    \draw [red, thick]
        (-3,2) -- (-2,2) -- (-2,3) -- (-3,3) -- (-3,2);

      %Second half:
\begin{scope}[shift={(0.25,1)}]
    \draw[thick] 
        (0,3) -- (3,0);
    \draw[thick]
        (0,3) -- (0,0);
    \draw[thick] 
        (3,3) -- (3,0);

    \draw[thick]
        (0,0) -- (3,0);
    \draw[thick] 
        (0,3) -- (3,3);
    \filldraw[green] (0,0) circle (2pt);
    \filldraw[blue] (1,0) circle (2pt);
    \filldraw[blue] (2,0) circle (2pt);
    \filldraw[green] (3,0) circle (2pt);
    \filldraw[blue] (0,1) circle (2pt);
    \filldraw[blue] (1,1) circle (2pt);
    \filldraw[blue] (2,1) circle (2pt);
    \filldraw[blue] (3,1) circle (2pt);
    \filldraw[blue] (0,2) circle (2pt);
    \filldraw[blue] (1,2) circle (2pt);
    \filldraw[blue] (2,2) circle (2pt);
    \filldraw[blue] (3,2) circle (2pt);
    \filldraw[green] (0,3) circle (2pt);
    \filldraw[blue] (1,3) circle (2pt);
    \filldraw[blue] (2,3) circle (2pt);
    \filldraw[green] (3,3) circle (2pt);
    \filldraw[red] (1.5,0) circle (2pt);
    \filldraw[red] (1.5,1.5) circle (2pt);
    \filldraw[red] (1.5,3) circle (2pt);
    \filldraw[red] (0,1.5) circle (2pt);
    \filldraw[red] (3,1.5) circle (2pt);
\end{scope}
    
\end{tikzpicture}

\caption{Lifts of the lines determining $\mathcal{S}_A$ for $A = \{e_1, e_2, e_1 + e_2\}$ are depicted on the left with a fundamental domain outlined in red. On the right, the corresponding hyperplane arrangement on the torus is drawn where $L_2$ is red, $L_3$ is blue, and $L_2 \cap L_3$ is green. The top dimensional strata do not contain points in $L_2$.} 
\label{fig:p1}
\end{figure}
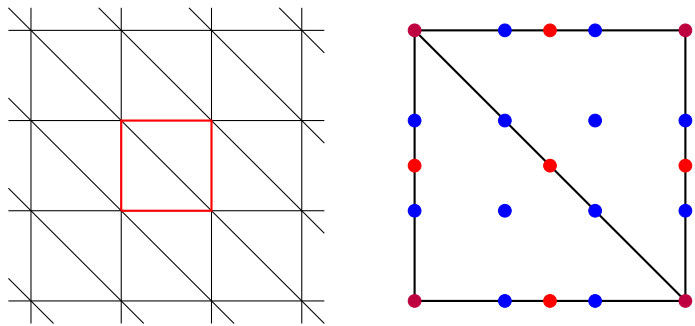

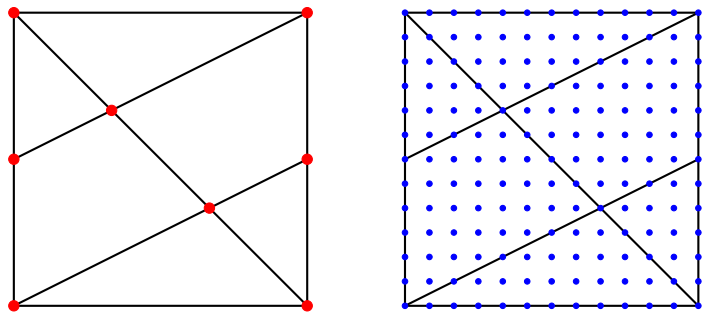
\begin{figure}[!b] 
\centering 

\begin{tikzpicture}[scale=1.3]
    \draw[thick]
        (-4,3) -- (-1,0);
    \draw[thick]
        (-4,3) -- (-4,0);
    \draw[thick]
        (-1,3) -- (-1,0);
    \draw[thick]
        (-4,0) -- (-1,1.5);
    \draw[thick]
        (-4,0) -- (-1,0);
    \draw[thick]
        (-4,3) -- (-1,3);
    \draw[thick]
        (-4,1.5) -- (-1,3);
    \filldraw[red] (-4,0) circle (1.5pt);
    \filldraw[red] (-4,1.5) circle (1.5pt);
    \filldraw[red] (-1,1.5) circle (1.5pt);
    \filldraw[red] (-1,3) circle (1.5pt);
    \filldraw[red] (-1,0) circle (1.5pt);
    \filldraw[red] (-4,3) circle (1.5pt);
    \filldraw[red] (-2,1) circle (1.5pt);
    \filldraw[red] (-3,2) circle (1.5pt);

    \draw[thick]
        (0,3) -- (3,0);
    \draw[thick]
        (0,3) -- (0,0);
    \draw[thick]
        (3,3) -- (3,0);
    \draw[thick]
        (0,0) -- (3,1.5);
    \draw[thick]
        (0,0) -- (3,0);
    \draw[thick]
        (0,3) -- (3,3);
    \draw[thick]
        (0,1.5) -- (3,3);
    \foreach \x in {0, 1/4, 2/4, 3/4, 1, 5/4, 6/4, 7/4, 2, 9/4, 10/4, 11/4, 3} {

    \foreach \y in {0, 1/4,1/2,3/4,1,5/4, 3/2, 7/4, 2, 9/4, 5/2, 11/4, 3} {
    \filldraw[blue] (\x,\y) circle (0.8pt);
    }
}

\end{tikzpicture}

\caption{The stratification $\mathcal{S}_A$ on the torus with $A = \{e_1, e_2, e_1 + e_2, -e_1 + 2e_2\}$ is depicted on the left. The zero-dimensional strata are red dots, the one-dimensional strata are solid black lines, and the two-dimensional strata are white regions. On the right, we have the same stratification with $L_{12}$ drawn in blue. There are points in $L_{12}$ in all strata.}
\label{fig:morecomplicated}
\end{figure}

\clearpage 

\printbibliography

\Addresses

\end{document}